\documentclass{amsart}
\usepackage{epstopdf}
\usepackage[caption=false]{subfig}

\usepackage[numbers,sort&compress]{natbib}
\bibpunct[, ]{[}{]}{,}{n}{,}{,}
\makeatletter
\def\NAT@def@citea{\def\@citea{\NAT@separator}}
\makeatother

\theoremstyle{plain}
\newtheorem{teo}{Theorem}
\newtheorem{lema}[teo]{Lemma}
\newtheorem{cor}[teo]{Corollary}
\newtheorem{prop}[teo]{Proposition}

\theoremstyle{definition}
\newtheorem{defi}[teo]{Definition}
\newtheorem{ejem}[teo]{Example}

\theoremstyle{remark}
\newtheorem{nota}{Remark}

\usepackage{xcolor}
\usepackage{url}

\newcommand{\0}{\mathbf{0}}
\newcommand{\bbar}{\overline}
\newcommand{\C}{\mathbb{C}}
\newcommand{\ceil}[1]{\mathop{\lceil #1 \rceil}}
\newcommand{\dd}{\mathrm{d}}

\newcommand{\diag}{\mathrm{diag}}
\newcommand{\disc}{\mathrm{disc}}
\newcommand{\E}{\mathbb{E}}
\newcommand{\HH}{\mathbb{H}}

\newcommand{\I}{\mathbf{i}}
\newcommand{\inv}{^{-1}}
\newcommand{\J}{\mathbf{j}}
\newcommand{\K}{\mathbf{k}}
\newcommand{\mean}[1]{M( #1)}

\newcommand{\R}{\mathbb{R}}
\newcommand{\SF}{\mathbb{S}}
\newcommand{\Trace}{\mathop{\mathrm{Trace}}}
\newcommand{\U}{\mathbf{u}}
\newcommand{\V}{\mathbf{v}}

\newcommand{\vol}{\mathrm{Vol}}
\newcommand{\vv}{\vert}
\newcommand{\W}{\mathbf{w}}

\begin{document}


\title[Rayleigh quotient and left eigenvalues]{Rayleigh  quotient  and left eigenvalues of  quaternionic matrices}

\author[E.~Mac\'ias-Virg\'os, M.J.~Pereira-S\'aez, Ana D.~Tarr\'io-Tobar]{
	E.~Mac\'ias-Virg\'os,
	\address{Instituto de Matem\'aticas, Universidade de Santiago de 			Compostela, 15782-Spain}
	M.J.~Pereira-S\'aez
	\address{Facultade de Econom\'{\i}a e Empresa, Universidade da 			Coru\~na, 15071-Spain}
	\and
	Ana D.~Tarr\'io-Tobar
	\address{E.U. Arquitectura T\'ecnica, Universidade de A Coru\~na, 15008-		Spain}
	}

	\thanks{$^\ast$
	{\tt quique.macias@usc.es, maria.jose.pereira@udc.es, {ana.dorotea.tarrio.tobar@udc.es}}
	}
	
\begin{abstract}  

We study the Rayleigh quotient of a Hermitian matrix with quaternionic coefficients and prove its main properties.
As an application,
we give some relationships between left and right eigenvalues of Hermitian and symplectic matrices. \end{abstract}

\maketitle




\section{Introduction}
The Rayleigh  quotient  of a matrix, introduced by the British physicist Lord Rayleigh in 1904 in his book   ``The theory of sound'', is a well known tool  which is widely used to obtain estimates of  the eigenvalues of real and complex matrices (\cite{HORN,BHATIA}).

For  quaternionic matrices, however,  only a few references about the Rayleigh  quotient  can be found in the literature, and there is a lack of a general exposition of its properties and main results.
Since  quaternions have many applications, among them in quantum
mechanics, solid body rotations, and signal theory
(\cite{BADENSKA}), it seems useful to fill that gap. Notice that most of the results and proofs are analogous to the complex case, but they refer to {\em right} eigenvalues.

On the other hand, very little is known about {\em left} eigenvalues of  quaternionic matrices. Wood (\cite{WOOD}) proved  that every   quaternionic matrix has
at least one left eigenvalue. Huang and So {(\cite{HUANG-SO})} completely solved the case of $2\times 2$ matrices. The authors (\cite{MP2009,MP2010}) applied Huang and So's results to $2\times 2$ symplectic matrices. {The case $n=3$ was studied by So (\cite {SO}) and the authors (\cite{MP2014})}. Finally, Zhang (\cite{ZHANG2007}) and Farid, Wang and Zhang (\cite{FARID}) gave several Ger\v{s}gorin type theorems for   quaternionic matrices.

Consequently, many problems still remain open, in particular those about the relationship between left and right eigenvalues. In this paper we give some partial answers to this question, for Hermitian and symplectic matrices, as an application of the previously proved properties of the Rayleigh  {quotient.}

The contents of the paper are as follows. In Section \ref{SECTPRELIM} we {present} some preliminaries about the right eigenvalues of a  quaternionic matrix. 

In Section \ref{SECTquotient}, we consider a Hermitian  quaternionic $n\times n$ matrix $S$, and we define its Rayleigh  quotient  $h_S$ as a real function defined on the sphere $S^{4n-1}$. We compute the gradient, the Hessian  and the mean value of $h_S$, and we prove its main properties, among
them the min-max principle for eigenvalues (Section \ref{MINMAXSECT}).  

In Section \ref{SECTLEFT}, we introduce left eigenvalues and we study the case $n=2$ with some detail, as a testing bench for later results. {For an arbitrary Hermitian matrix $S$, our main result (Theorem \ref{HERMITMAIN}) is that  the real part of any left eigenvalue $\lambda$  is bounded by the right eigenvalues, in a way that depends on the dimension of the $\lambda$-eigenspace.} 
{As we shall see, this implies that the existence of left eigenvalues with a {high-dimensional} space of eigenvectors depends on the mutiplicity of the right eigenvalues.}

Finally, in Section \ref{SECTSYMPL} we state similar results for symplectic matrices.

{Our results suggest that there are still many more hidden relations between left and right eigenvalues.}


\section{Preliminaries}\label{SECTPRELIM}

As a general reference for  quaternionic linear algebra we take Rodman's book \cite{RODMAN}. For a brief survey on {quaternions} and matrices of  quaternions, see Zhang's paper \cite{ZHANG1997}.

\subsection{Basic notions}
We denote by $\HH$ the non-commutative algebra of  quaternions. For the  quaternion $ q\in\HH$ we {denote} its conjugate by $\bbar  q$, its norm by {$\vv q \vv$} and its real part by $\Re( q)$.

Let $\HH^{n\times n}$ be the space of $n\times n$ matrices with  quaternionic coefficients. If $M\in\HH^{n\times n}$, we denote by $M^*$ its conjugate transpose {$(\bbar M)^T$}. The  quaternionic space of $n$-tuples {$\U=(u_1,\dots,u_n)^T$}, with $u_i\in\HH$, will be denoted by $\HH^n$. We shall always consider it as a {\em right} vector space over $\HH$, endowed with the {\em Hermitian} product $\langle \U,\V\rangle =\U^*\V$. Notice that $\vert \U \vert^2 =\langle \U,\U\rangle $ is the Euclidean norm in $\R^{4n}$, so the {\em scalar} product is 
$\U\cdot \V=
\Re \langle \U,\V\rangle$.

The matrix {$S\in\HH^{n\times n}$} is {\em Hermitian} if it  is self-adjoint for the Hermitian product, that is, $\langle S\U,\V\rangle =\langle \U,S\V\rangle $, or equivalently, $S^*=S$. The matrix {$A\in\HH^{n\times n}$} is {\em symplectic} if the associated linear map preserves the Hermitian product, that is, $\langle A\U,A\V\rangle =\langle \U,\V\rangle $, or equivalently, $A^*A=AA^*=I_n$.

Two  quaternions $q$ and $q'$ are {\em similar} if there exists some $ r \in \HH$, $ r \neq  0$, such that $q'= r q  r\inv$. Equivalently, they have the same norm and the same real part, that is, $\vert q \vert=\vert  q' \vert$ and  $\Re(q)=\Re(q')$ (\cite[Theorem 2.2]{ZHANG1997}).  As a consequence, any  quaternion $q$ is similar to a complex {number}, namely $z=t+s\,\I$, where $t=\Re(q)$ and $t^2+s^2=\vert q \vert^2$. Notice that {$z\in\C$} and its conjugate $\bbar z$ are similar  quaternions.

{Finally, we shall need the following result, which can be proved by a direct computation
\begin{lema}\label{CONMUT}{Let $\omega,\omega^\prime$ be two  quaternions such that $\vv \omega \vv=\vv \omega^\prime \vv=1$
and $\Re(\omega)=\Re(\omega^\prime)=0$. If $\omega\omega^\prime=\omega^\prime \omega$ then $\omega=\pm\omega^\prime$.}
\end{lema}}

\subsection{Right eigenvalues}

The theory of right eigenvalues is well known, and has many properties in common with the complex case.

\begin{defi}
The  quaternion $q\in \HH$ is a {\em right eigenvalue} of the  matrix $M\in\HH^{n\times n}$ if there exists some {vector} $\U\in\mathbb{H}^n$, $\U\neq \0$, such that $M\U=\U q$.  
\end{defi}

Notice  that the {eigenvectors associated to  a right eigenvalue do not form a  vector subspace}. Instead,
right eigenvalues are organized in similarity classes.

\begin{prop}\label{CHANGESIMIL}Let  $q$ be a {right eigenvalue of $M\in \HH^{n\times n}$,} and {let $\U$  be} a  $q$-eigenvector. If  $r\in \HH$ is a non-zero  quaternion,  then $rq r\inv$ is {also a right eigenvalue of $M$,} and  $\U r\inv$ is an $rq r\inv$-eigenvector.
\end{prop}
\begin{proof}
$M(\U r\inv )= (\U q) r\inv= \U r\inv (rq r\inv)$.
\end{proof}

So, the computation of the right eigenvalues of the matrix $M$ is reduced to compute the complex representatives of their similarity classes. This can be done as follows.
Each  quaternion $ q$ can be written in a unique form as $ q =u+\J v$, with $u,v\in \C$ complex numbers. Then the matrix $M\in \HH^{n\times n}$ decomposes as $M=U+\J V$, with $U,V\in \C^{n \times n}$ complex matrices. We define the associated complex matrix 
$$c(M)=\begin{bmatrix} U&-\bbar V\,\cr V\,& \bbar U\cr\end{bmatrix}.$$

It is straightforward to verify that he map $c\colon \HH^{n\times n}\to \C^{2n \times 2n}$ is an injective morphism of $\R$-algebras and satisfies $c(M^*)=c(M)^*$. 

\begin{prop} The right eigenvalues of $M$ are grouped in $n$ similarity classes $[z_1],\dots,[z_n]$. The complex representatives $z_1,\bbar z_1,\dots, z_n,\bbar z_n$ are the eigenvalues of the complex matrix $c(M)$.
\end{prop}

See  \cite[Theorem 5.5.3]{RODMAN} for a discussion of the Jordan form of $M$.

\begin{ejem}\label{EXRIGHT}Let the matrix $M=\begin{bmatrix}
0\ &\J \cr 
\I  & 0\cr
\end{bmatrix}$.
The eigenvalues of $c(M)$ are $z_1=\frac{1}{\sqrt{2}}(1+\I)$, $z_2=\frac{1}{\sqrt{2}}(-1+\I)$ and their conjugates $\bbar z_1$,$\bbar z_2$. Then, the right eigenvalues of $M$ are all the  quaternions $q$  such that $\Re(q)=\pm 1/\sqrt{2}$ and $\vert q\vert =1$.
\end{ejem}

Notice that, unlike the usual complex case, the matrices $M-q I$ in the latter example  are  invertible, for all right eigenvalues $q$. This can be easily seen by computing their kernel. This leads to the notion of {\em left} eigenvalue, {as a  quaternion $\lambda\in\HH$ such that $M-\lambda I$ is not invertible (see Section \ref{SECTLEFT})}.


\section{The Rayleigh  quotient }\label{SECTquotient}

{The well known Rayleigh  quotient   for complex matrices can be generalized to matrices of  quaternions. We now focus on Hermitian matrices.
}

\subsection{Definition and first properties}
Recall that the $n\times n$  quaternion  matrix $S$ is {\em Hermitian} if $S=S\sp*$.
As it is well known, any {\em right} eigenvalue $q$ of $S$  is real: in fact, if $S\U=\U q$ then
 $$\vert \U\vert^2q=\U^*\U q= \U^*S\U={(\U^*S\U)^*}$$ is real,
 hence $\bbar q=q$. Moreover, $S$ is diagonalizable \cite[Theorem 5.3.6]{RODMAN}. Let $ t_1\leq\dots\leq t_n$ be {the   eigenvalues} of $S$, and 
let $\U_1,\dots,\U_n$ be an orthonormal basis of eigenvectors. Then $S$ diagonalizes
as $S=U\diag[ t_1,\dots, t_n]U^*$,
where $U$ is a symplectic matrix whose columns are the $\U_j$'s. 



\begin{defi}If $S$ is a Hermitian $n\times n$ matrix, and $\V\in \HH^n$ is a vector, $\V\neq \0$, the {\em Rayleigh  quotient } is the real number
$$R(S,\V)=\frac{\V^*S\V}{\vert \V\vert^2}.$$
\end{defi}
\begin{prop}
 If $ t\in\R$ is an eigenvalue of $S$, and  {$\U\in\HH^n$} is a  $ t$-eigenvector, then
$R(S,\U)= t$.
\end{prop}

If $\V\in\HH^n$ is a vector, $\V\neq \0$, {we can} write it in coordinates with respect to the orthonormal basis $\{\U_j\}_{j=1,\dots,n}$ as
$$\V={\sum_{j=1}^n \U_jx_j},  \quad  x_j\in\HH.$$

\begin{prop}\label{WEIGHT}
The Rayleigh  quotient  equals the weighted mean
$$R(S,\V)=\frac{\sum_j  t_j\vert x_j\vert^2}{\sum_j \vert x_j\vert^2},$$
where {the real numbers $ t_j$ are the eigenvalues of $S$}. 
\end{prop}
\begin{proof}
{Since $\U_i^*\U_j=\delta_{ij}$}, we have
$$\vert \V \vert^2=\V^*\V=(\sum_i \bbar x_i \U_i^*)(\sum_j \U_jx_j)=\sum_{i,j}\bbar x_i\U_i^*\U_jx_j=\sum_j\bbar x_j x_j=\sum_j\vv x_j\vv^2.$$
Analogously, since $S\U_j=\U_j t_j$,
$$\V^*S\V=(\sum_i\bbar x_i \U_i^*)(\sum_j\U_j t_jx_j)=
\sum_j\bbar x_j t_jx_j=\sum_j t_j \vv x_j\vv^2. \qedhere$$
\end{proof}
By diagonalization, we have reduced the problem of an arbitrary bilinear form to the corresponding quadratic form.
As a consequence we have:
\begin{prop}\label{MINMAXPROP}{The minimum value of the function $R(S,\V)$, defined in $\HH^n\setminus{\{\0\}}$,  is the lowest {eigenvalue of $S$}.  The
maximum value {is} the {highest eigenvalue of $S$}}.
\end{prop}
\begin{proof}
First, notice that $R(S,\V/\vert\V\vert)=R(S,\V)$, so we can assume that $\vert \V \vert=1$. It is known \cite[Appendix B]{HORN} that a convex function
$$\sum\limits_{j=1}^n  t_j s_j,   \quad {\text{with\ }}\sum\limits_{j=1}^n s_j=1, s_j\geq 0,$$
attains its maximum value at $ t_n=\max  t_j$ and its minimum value at $ t_1=\min  t_j$. {By taking into account Proposition \ref{WEIGHT}, the result follows}.
\end{proof}
In fact, {we shall  compute} all the critical values of the function. This is a variational characterization of eigenvalues, which was discovered in {connection} with problems of physics \cite{HORN,BHATIA}. 

\begin{prop}\label{CRITICAL}The critical values of the function $R(S,\V)$ are the {right} eigenvalues $ t_j$ of $S$. The index of $ t_j$ equals $\sum\nolimits_{i<j} l_i$, where $l_i$ is the multiplicity of $ t_i$. 
\end{prop}
\begin{proof}
{Even though it is possible to do a direct computation by using coordinates, we shall give a more synthetic proof.}

By differentiating the function {$R=R(S,\V)\colon \HH\setminus \{\0\} \to \R$} we obtain
\begin{align*}
R_{*\V}(\W)&=\frac{1}{\vv \V \vv^4}\left((\W^*S\V+\V^*S\W)\vv\V\vv^2-(\V^*S\V)(\W^*\V+\V^*\W)\right)\\
&=\frac{2}{\vv\V\vv^4}\left(\Re(\V^*S\W)\vv \V \vv^2-(\V^*S\V)\Re(\V^*\W)\right) \\
&=\frac{2}{\vv \V\vv^4}\left((S\V\cdot\W)\vv \V\vv^2-(\V^*S\V)( \V\cdot\W)\right)\\
&=\frac{2}{\vv \V \vv^2}\left(( S\V\cdot \W)-R(S,\V)( \V\cdot \W)\right)\\
&=\frac{2}{\vv \V \vv^2}\left( S\V-R(S,\V)\V\right)\cdot\W,
\end{align*}
{where   $(\cdot)$ represents the scalar product} in $\R^{4n}$.

Hence the gradient of $R(S,\V)$ is
$$G_\V=\frac{2}{\vv \V \vv^2}\left(S\V-R(S,\V)\V\right).$$


{Since   $R(S,\V/\vv\V\vv)=R(S,\V)$}, we can assume that $\vv \V \vv=1$. 
Let us denote by $h$ the restriction of the Rayleigh function to the sphere $S^{4n-1}\subset \HH^n$ of unitary vectors. We check that
$G_\V \perp \V$, because
$${\frac{1}{2}}\langle \V,G_\V \rangle=\V^*(S\V-R(S,\V)\V)=\V^*S\V-R(S,\V)\vv \V \vv^2=0.$$
Hence, $\V\cdot G_\V=\Re\langle \V,G_\V \rangle=0$ and $G_\V$ is tangent to the sphere for the scalar product, so it is also the gradient of the restriction $h$.

It follows that  the point $\V$ is critical (both for the function and its restriction) if and only if $S\V=R(S,\V)\V$,
that is, $\V$ is an eigenvector of the eigenvalue {$R(S,\V)= t\in\R$}.

This will allow {us} to compute the Hessian
\begin{align*}
H_\V(\W)&=(G_\V)_*(\W)\\
& =\frac{2}{\vv \V \vv^4}{\big(S\W-(R_{*\V}(\W)\V+R(S,\V)\W)\vv \V \vv^2}\\
& \quad  - (S\V-R(S,\V)\V)(\W^*\V+\V^*\W) \big) \\
&={\frac{2}{\vv \V \vv^2}}\big(S\W-(R_{*\V}(\W)\V+R(S,\V)\W)\big)\\
&={\frac{2}{\vv \V \vv^2}}\big(S\W-R(S,\V)\W\big)\\
&={
\frac{2}{\vv \V \vv^2}(S\W- t\W)=\frac{2}{\vv \V \vv^2}(S- tI)\W.
}
\end{align*}

Moreover, if $\vv\V\vv=1$  and $\W\in T_\V{S^{4n-1}}$, that is,   $\V\cdot\W= \Re(\V^*\W)=0$, then 
{$$\V\cdot H_\V(\W)=\Re  (\V^*H_\V(\W))=2\Re(\V^*S\W)=2\Re(\W^*S\V)  
=2\W\cdot \V t=0.$$} 
Hence $H_\V(\W)\in T_\V{S^{4n-1}}$,
and  $H_\V(\W)$ is also the Hessian of the restriction $h$.

Now, we compute the index of $ t=R(S,\V)$, which  is the number of negative eigenvalues of the Hessian at the critical point $\V$.  
If  $\mu$ is an eigenvalue of the Hessian, we have  
$H_\V(\W)=\W\mu$, for some $\W\neq \0$,
that is,
$$2(S- t I)\W=\W\mu,$$
so we are looking for the eigenvalues $\mu$ of of the {\em shifted} matrix  {$S- t I$}. Hence, the eigenvalues of the Hessian are $\mu_k= 2(t_k- t)$, {twice} the differences  with the other {right} eigenvalues $ t_k$ of  $S$, and the result follows.
\end{proof}
For instance, the {minimum} $ t_1$ has index $0$. The maximum $ t_n$ has index $n-l_n$.

\subsection{The min-max principle for eigenvalues}\label{MINMAXSECT}
As in the complex case, it is possible to refine Proposition \ref{CRITICAL}. Now, we constrain $\V$ to a $k$-dimensional subspace, in order to obtain a  quaternionic version of the the so-called  min-max  Courant-Fischer-Weyl  theorem (\cite[Theorem 4.2.6]{HORN}).

Fix some $k\in\{1,\dots,n\}$ and let $\E\subset \HH^n$ be any {$\HH$-subspace} of dimension $k$. We shall denote $\E^*=\E\setminus \{\0\}$.

Let $\{\U_1,\dots,\U_n\}$ be again an orthonormal basis of eigenvectors. We have that
$$\E \cap \langle \U_k,\dots \U_n \rangle \neq \0,$$
due to dimension reasons. Then there exists $\V=\sum_{j=k}^n\U_ix_j\in\E^*$, and its Rayleigh  quotient  is
$$R(S,\V)=
\frac{\sum_{j=k}^n t_j\vert x_j\vert^2}{\sum_{j=k}^n\vert x_j\vert^2}\geq   t_k,$$
because $ t_j\geq   t_k$ for all $j\geq  k$.

This implies that
$$M_\E:=\max\{R(S,\V)\colon \V\in \E^*\}\geq   t_k.$$
Since this is true for all $\E$ we conclude that
$\min\{M_\E\colon \dim \E=k\}\geq   t_k$.
This is in fact an equality:
\begin{teo}\label{MINMAXTEO}$ t_k=\min\{M_\E\colon \dim\E=k\}$.
\end{teo}\begin{proof}
It only remains to prove that $M_\E\leq  t_k$ for some $\E$ with $\dim \E=k$. Take $\E=\langle \U_1,\dots,\U_k\rangle$. Then  for all $\V\in\E^*$ we have
$$R(S,\V)=\frac{\sum_{j=1}^k t_j\vert x_j\vert^2}{\sum_{j=1}^k \vert x_j\vert^2}\leq  t_k,$$
because $ t_j\leq  t_k$ if $j\in\{1,\dots,k\}$.
\end{proof}

Analogously, if we denote $$m_\E{:=}\min\{R(S,\V)\colon \V\in\E^*\},$$
 we have
\begin{cor}\label{MAXMINCOR}
 ${t_{n-k+1}}=\max\{m_\E\colon \dim\E=k\}$.
\end{cor}
The proof is immediate if we take into account that
$  t_{n-k+1}(S)=-   t_k(-S)$.

As {particular cases}, for $k=1,n$ we have Proposition \ref{MINMAXPROP}.

\subsection{Mean value} We want to compute the mean value of the Rayleigh function over the sphere $\SF^{N-1}\subset 
\HH^n$,  where $N=4n$. 

Let $h\colon \SF^{N-1}\to \R$ be  the restriction given by $h(\V)= \V^* S \V$, $ \vert \V \vert=1$.
In order to compute the mean value  
of $h$, $$\mean{h}=\frac{1}{\vol (\SF^{N-1})}\int\nolimits_{\SF^{N-1}}\V^*S\V\, \dd \V,$$  one can consider hyper-spherical coordinates and {to undertake a long direct computation}. Another proof follows {by using} Pizzetti's formula ({\cite[Formula (11.3)]{GRAY}}). However, in order to have a similar result for the variance, we shall use  {\em moments}, as explained in 
 {Gray's book \cite[Appendix A.2]{GRAY}.}

{For any integrable function
$F=F( u_1,\dots, u_N)\colon \R^N \to \R$ we denote by $\mean{F}$ the average of $F$ over the unit sphere $\SF^{N-1}\subset \R^N$. 

{\begin{lema}\label{PARTICULAR} {\cite[Theorem 1.5]{GRAY}.}
$$\mean{ u_i^2}=\frac{1}{N},  \quad  \mean{ u_i^4}=\frac{3}{N(N+2)}, \quad  \mean{ u_i^2 u_j^2}=\frac{1}{N(N+2)} \text{\ if\ } i\neq j.$$
\end{lema}
}

\begin{teo}The expected value of  the Rayleigh  quotient  $R(S,\V)$ over the sphere $\SF^{4n-1}\subset \HH^n$ equals
$$\mean{h}=\frac{1}{n}\Trace S.$$
\end{teo}

\begin{proof}
According to Proposition \ref{WEIGHT}, $h(\V)={\sum_{j=1}^n } t_j\vert x_j\vert^2$, where $x_j\in\HH$. Since each $x_j$ has four real coordinates, our function can be written as
$$h( u_1,\dots, u_N)= t_1( u_1^2+\cdots+ u_4^2)+\cdots+ t_n( u_{N-3}^2+\cdots+ u_N^2), \quad  N=4n,$$
where $ u_1^2+\cdots+ u_N^2=1$.

Then, by {Lemma} \ref{PARTICULAR},
	$$\mean{h} = 4 \frac{1}{N}( t_1\cdots+ t_n)=\frac{1}{n}( t_1+\cdots+ t_n),$$
{is the arithmetic mean of the eigenvalues} and the result follows.
\end{proof}
A similar computation gives us the relationship between the second central moment of $h$ and the variance {of} the eigenvalues. {We denote by 
$$\mu=\frac{1}{n}\sum_{i=1}^n t_i$$  
the mean of the eigenvalues and by 
$$\sigma^2=\frac{1}{n}\sum_{i=0}^n( t_i-\mu)^2$$ its variance.}

\begin{teo} {The second central moment of the Rayleigh quotient over the sphere {is proportional to} the variance of the eigenvalues,
$$\mean{(h-\mu)^2}=\frac{1}{2n+1}\sigma^2.$$}
\end{teo}

\begin{proof}
The proof follows from the well-known identity $\mean{(h-\mu)^2}=\mean{h^2}-\mu^2$,
and Lemma \ref{PARTICULAR}.
\end{proof}
}

\section{Left eigenvalues}\label{SECTLEFT}
As mentioned in Example \ref{EXRIGHT}, {the matrix $M-q I$ can be invertible} for a right eigenvalue $q$ of $M$. {This motivates the following definition.}

\begin{defi}
The  quaternion $\lambda\in \HH$ is a {\em left eigenvalue} of the  matrix $M\in\HH^{n\times n}$ if the matrix $M-\lambda I_n$ is not invertible. 
\end{defi}

{The existence of left eigenvalues for any  quaternionic matrix was proved by Wood in \cite{WOOD}.} {Notice that the left eigenvalues of a matrix are not invariant by a change of basis {\cite[Example 7.1]{ZHANG1997}.}}

If $\lambda$ is a left eigenvalue of $M$, the set of vectors $\V\in\HH^n$ such that $M\V=\lambda\V$
is a right $\HH$-vector subspace $V(\lambda)\neq \{\0\}$ of $\HH^n$.  A non-null element $\V\neq 0$ of $V(\lambda)$ is called a {\em $\lambda$-eigenvector}. By dividing it by its norm we can always assume that
 $\vv \V \vv=1$.

Clearly, if a right eigenvalue is a real number, then it is also a left eigenvalue. So {the  problem  is} to determine the non-real left eigenvalues of $M$, if any.

\subsection{$n=2$}
The case $n=2$ was completely solved by Huang and So in \cite{HUANG-SO}.
Let $M=\begin{bmatrix}a & b \cr c & d\cr\end{bmatrix}$. If  $b=0$ or $c=0$ then the left eigenvalues are $a,d\in\HH$, as it is straightforward to check by hand.
When $bc\neq 0$,  Huang and So gave explicit formulas for the left eigenvalues. In particular they proved (\cite[Theorems 2.3, {3.1 and 3.2}]{HUANG-SO}) the following result:

\begin{teo}\label{HUANGSO}
If $bc\neq 0$, 
\begin{enumerate}
\item
the left eigenvalues of $M$ are given by $\lambda=a+bx$, where $x$ is any solution of the quadratic equation 
$x^2+a_1x+a_0=0$,
with
$
a_1 = b^{-1}(a-d)$ and
$a_0 = -b^{-1}c$;
\item
the matrix $M$ has either one, two or infinite left eigenvalues; 
\item
the infinite case happens if and only if
$a_0\in\R$, $a_1\in\R$,  and
$\Delta=a_1^2-4a_0<0$;
\item
in the latter case, the left eigenvalues can be written as
\begin{equation}\label{EIGEN}
\lambda= \frac{1}{2}(a+d+b\xi),  \quad  \Re(\xi)=0, \vert \xi\vert^2 = \vert \Delta\vert.
\end{equation}
\end{enumerate}
\end{teo}

\subsection{Left eigenvalues of $2\times 2$ Hermitian matrices}\label{LEFTHERMITE}
{We apply the previous results to the Hermitian case}.
If
$S$
is a $2\times 2$ Hermitian  matrix,  the condition $S=S^*$  means that $S=\begin{bmatrix}s & b \cr b^* &  s^\prime\cr\end{bmatrix}$, where $s, s^\prime\in\R$.  If $b=0$,  the matrix $S$ has {only} two left eigenvalues, the real numbers $s, s^\prime$, {which are also right eigenvalues.}

\begin{prop}\label{MAINDOS}
{When $b\neq 0$, }
\begin{enumerate}
\item
the matrix $S$ has   two real eigenvalues (which may be different or not). They can be computed as the roots  of the real equation
\begin{equation}\label{MOOREQ}(s- t)( s^\prime- t)-\vert b\vert^2=0.
\end{equation}
\item
it has {also} non-real   left eigenvalues if and only if $\Re(b)=0$ and $s= s^\prime$. In this case, the left eigenvalues are given by the formula
\begin{equation}\label{TODOS}
\lambda=s+b\omega,  \quad  \Re(\omega)=0, \vert \omega\vert=1.
\end{equation}
\end{enumerate}

\end{prop}

\begin{proof}
1.  Notice that the discriminant of Equation \eqref{MOOREQ} is
$$\disc=(s+ s^\prime)^2-4(s s^\prime-\vert b\vert^2)=(s- s^\prime)^2+\vert b \vert^2\geq  0.$$ {It is easy to check that the two real roots  
 are  eigenvalues. }

2. Since {there are already two real eigenvalues, we only have to consider the infinite case of Theorem \ref{HUANGSO}.} If $b\neq0$, Huang-So's conditions are
\begin{align}
a_1=&b^{-1}(s- s^\prime)=\frac{b^*}{\vert b\vert^2}(s- s^\prime)\in\R,\label{A1}\\
a_0=&-b^{-1}b^*=-\frac{(b^*)^2}{\vert b \vert^2}\in\R \label{A2},\\
a_1^2-4a_0=&\frac{(b^*)^2}{\vert b \vert^2}\left[\frac{(s- s^\prime)^2}{\vert b\vert^2}+4\right]<0,\label{A3}
\end{align}
which imply $(b^*)^2\in\R$, by \eqref{A2}, and $(b^*)^2<0$,
by \eqref{A3}. 

Hence $b^2\in\R$ and $b^2<0$. 
This implies  $\Re(b)=0$, $b^*=-b$ and $b^2=-\vert b \vert^2$.   But then 
$$\Re(a_1)=\frac{s- s^\prime}{\vert b \vert^2}\Re(b^*)=0,$$
so $s- s^\prime=0$ by \eqref{A1}.

Since $a_1=0$, $a_0=1$ and $\Delta=a_1^2-4a_0=-4$,
 Formula \eqref{EIGEN} implies \eqref{TODOS}.
\end{proof}

\begin{nota}\label{SEPARO}
Notice that in the latter case, among the infinite left eigenvalues there are two real ones. {In fact, the two solutions of Equation \eqref{MOOREQ} are $t = s\pm \vv b \vv$. They
correspond to Formula \eqref{TODOS} with $\omega=\pm \frac{b}{\vert b \vert}$.}
\end{nota}

\begin{ejem}\cite[Example 2.5]{HUANG-SO} The matrix $S=\begin{bmatrix}0&1+\I\cr 1-\I&0\cr\end{bmatrix}$ has only two left eigenvalues, {$\lambda=\pm\sqrt{2}$}, which are also its right eigenvalues.
\end{ejem}

\begin{ejem}\cite[Example 5.3]{ZHANG1997} Let $S=\begin{bmatrix}0&\I\cr -\I&0\cr\end{bmatrix}$. The real eigenvalues are  $\pm 1$. The left eigenvalues are ${\lambda}=\I \,\omega$, where $\Re(\omega)=0$ and $\vert \omega\vert=1$, that is, ${\lambda} =t+y\J+z\K$, with {$t,y,z\in\R$,} $t^2+y^2+z^2=1$.
\end{ejem}

\subsection{Relationship with the Rayleigh  quotient }
The previous section gives an idea of the difficulty of computing the left eigenvalues of a given (Hermitian) matrix. {In fact, no general method is known}. In this section we shall give a new relationship between left and right eigenvalues.

Let $S$ be a Hermitian matrix.
Let $\lambda\in\HH$ be a left eigenvalue of $S$ and let $\V$ be a $\lambda$-eigenvector, with $\vv \V \vv=1$.  Then $R(S,\V)=\V^*\lambda \V$ is a real number.

Notice that {$\V^*\bbar\lambda \V=\V^*\lambda\V$} even if $\bbar \lambda\neq \lambda$.

\begin{lema}\label{REALPART}	$R(S,\V)=\Re(\lambda)$, the real part of $\lambda$.
\end{lema}

\begin{proof}If the coordinates of $\V$ with respect to an orthonormal basis are $x_1,\dots,x_n$, then
since $\V^*S\V=\V^*\lambda\V$ is real, we have
\begin{align*}\V^*\lambda\V&=\sum \bbar x_i \lambda x_i=\Re(\sum \bbar x_i \lambda x_i)=
\sum \Re(\bbar x_i \lambda x_i)\\
&=\sum \Re(\lambda x_i\bbar x_i)=\sum \Re(\lambda\vv x_i \vv^2)=\Re(\lambda)(\sum \vv x_i\vv^2)=\Re(\lambda).  \qedhere
\end{align*}
	\end{proof}

Then,  from Proposition \ref{MINMAXPROP} it follows that
\begin{prop}\label{SIMPLE}
If $\lambda$ is a left eigenvalue {of $S$} and $\V$ is a $\lambda$-eigenvector,  then
\begin{equation}\label{MINMAXZHANG}
 t_1\leq R(S,\V)=\Re(\lambda)\leq  t_n.
\end{equation}
\end{prop}

Next Theorem  refines the latter formula, as an application of the min-max theorems of Section \ref{MINMAXSECT}.  {It gives a new relationship between left and right eigenvalues.}

\begin{teo}\label{HERMITMAIN} Let $\lambda$ be a left eigenvalue of the Hermitian matrix $S$, with real eigenvalues $ t_1\leq\dots\leq t_n$. If  the $\lambda$-eigenspace $V(\lambda)$ verifies 
 $\dim V(\lambda)\geq  k$ then $$ t_k \leq \Re(\lambda) \leq  t_{{n-k+1}}.$$
\end{teo}

\begin{proof}
Let $\E=V(\lambda)$. By Lemma \ref{REALPART}, the Rayleigh function is constant on $\E$, so $m_\E=\Re(\lambda)=M_\E$. 
Let $\dim V(\lambda)=j\geq  k$, then, by Theorem \ref{MINMAXTEO} and Corollary \ref{MAXMINCOR}, we have 
$$ t_k \leq  t_j \leq \Re(\lambda) \leq  t_{n-j+1} \leq  t_{n-k+1}. \qedhere$$  
\end{proof}

Notice that the inequality \eqref{MINMAXZHANG} is a particular case, since $\dim V(\lambda)\geq  1$.

{\begin{nota}
Using orthonormal coordinates one can prove that 
$\vert \lambda \vert \leq \max \vert  t_i\vert$,
where the right term is the (right) spectral radius of the Hermitian matrix \cite{ZHANG2007}.
\end{nota}}

\begin{ejem}For the Hermitian matrices $S=\begin{bmatrix}s & b \cr b^* & s\cr\end{bmatrix}$  with non-real left eigenvalues {(that is, with $\Re(b)=0$, see Section \ref{LEFTHERMITE}),} we have that $\Re(\lambda)=s$ by \eqref{TODOS}, and $\vv \lambda \vv^2=s^2 + \vv b \vv^2$, while $ t_1=s-\vv b\vv$ and $ t_2=s+\vv b \vv$.
\end{ejem}

{The next Corollary shows the influence of the left eigenvalues on the right ones.}

{
    \begin{cor}\label{SIZE}Assume that the $n\times n$ Hermitian matrix $S$  has a left eigenvalue $\lambda$ such that $k=\dim V(\lambda)> \ceil{n/2}$. Then,  the right eigenvalues $ t_{n-k+1}=\dots= t_k$ have multiplicity $l_k\geq  2k-n$ and they are equal to $\Re(\lambda)$.
    \end{cor}
}

{
    \begin{proof}
 By Theorem \ref{HERMITMAIN}, we have $ t_k\leq \Re(\lambda)\leq  t_{n-k+1}$. Moreover, $k>\ceil{n/2}$ implies $n-k+1\leq k$, hence $ t_{n-k+1}\leq  t_k$. That means that $ t_{n-k+1}= t_k$,
    hence $\Re(\lambda)= t_j= t_k$, for all $n-k+1\leq j \leq k$. This implies that the multiplicity $l_k$ of $ t_k$ is at least $2k-n$. 
    \end{proof}
}

    \begin{ejem}Let $S$ be a Hermitian matrix of order $5$, diagonalizable to $\diag[ t_1,\dots, t_5]$. Assume that $S$ has some left eigenvalue $\lambda$ with $\dim V(\lambda)=4$. Then $\Re(\lambda)= t_4= t_3= t_2$ has at least multiplicity $3$.
    \end{ejem}

    \begin{ejem}Let $S$ be a Hermitian matrix of order $6$, diagonalizable to $\diag[ t_1,\dots, t_6]$. Assume that $S$ has some left eigenvalue $\lambda$ with $\dim V(\lambda)=4$. Then $\Re(\lambda)= t_4= t_3$ has at least multiplicity $2$.
    \end{ejem}


\section{Symplectic matrices}\label{SECTSYMPL}
In this section we extend our results to symplectic matrices {with quaternionic coefficients}.

Recall that the $n\times n$ matrix $A$ is {\em symplectic} if $A^*A=I_n$. Its right eigenvalues have norm $1$, because if {$q$ is a right eigenvalue}, $A\U=\U q$, with $\U\neq \0$, then
$$\vv \U \vv^2=\langle \U,\U\rangle=\langle A\U,A\U\rangle =\langle \U q,\U q\rangle =\bbar q \U^*\U q=\vv q\vv^2 \vv \U\vv^2,$$
so $\vv q\vv=1$. Moreover, the matrix is diagonalizable \cite[Theorem 5.3.6]{RODMAN}.

Analogously, the left eigenvalues also have norm 1, because {if $\lambda$ is a left eigenvalue,} $A\V=\lambda\V$, with $\V\neq\0$,  {then} $$\vert \V \vert^2=\langle \V,\V\rangle=\langle A\V,A\V\rangle =\langle \lambda\V,\lambda\V\rangle =\V^*\bbar \lambda \lambda \V={\vv \lambda \vv^2 \, \vv \V \vv^2,}$$
{hence $\vv \lambda \vv=1$.}

\subsection{$n=2$}For $n=2$, the authors completely characterized in \cite{MP2009} the symplectic matrices which have an infinite number of left eigenvalues. 

\begin{teo}\label{DOSPORDOS}
 The only $2\times 2$ symplectic matrices with an infinite number of left eigenvalues are those of the form
\begin{equation}\label{MATRIX}
\begin{bmatrix}
r\cos\theta & -r\sin\theta\cr
r\sin\theta &  r\cos\theta \cr
\end{bmatrix},
 \quad     \quad  r\in\HH, \vert r \vert=1,  \quad  \sin\theta\neq 0.
\end{equation}
\end{teo}

\begin{prop}\label{SYMPLCASE}
For the matrix  in \eqref{MATRIX}:
\begin{enumerate}
\item
{The right eigenvalues are the similarity classes of 
{$q=r(\cos\theta\pm \sin\theta\, \rho)$}, where
$\rho$ is {any of the quaternions} such that  
$\Re(\rho)=0$, $\vv \rho \vv=1$, and $r=s+t\rho$, with $s,t\in \R$.}

\item
The left eigenvalues are $\lambda=r(\cos\theta+\sin\theta\,\omega)$, where $\omega$ is an arbitrary  quaternion such that
$\Re(\omega)=0$ and $\vert \omega\vert=1$. \end{enumerate}
\end{prop}

\begin{proof} {Part (1) follows from definition, by checking the eigenvectors $(\pm \rho,1)^T$, and taking into account that $r$ and $\rho$  commute, and that $\rho^2=-1$. 
}

Part (2) follows  from Proposition \ref{MAINDOS}.
{Notice that   there are two left eigenvalues which are right eigenvalues.}
\end{proof}

\subsection{Rayleigh  quotient  of a symplectic matrix}

The Rayleigh  quotient  can be defined for any non-Hermitian matrix. Let us assume that $A$ is symplectic.

\begin{defi} The Rayleigh  quotient  of $A$ is the real function 
$$R(A,\V)=\frac{\Re(\V^*A\V)}{\vv \V \vv^2}.$$
\end{defi}
As before, we shall consider its restriction $h_A$ to the sphere $S^{4n-1}\subset \HH^n$.

\begin{lema}Let $q$ be a right eigenvalue of the symplectic matrix $A$. If $\U$ is a  {$q$-eigenvector} then  $\bbar q$ is {a right eigenvalue} of $A^*$, and $\U$ is a $\bbar q$-eigenvector.
\end{lema}

\begin{proof}We have
$$A\U=\U q \Rightarrow \U=A\inv (\U q)=(A^*\U)q \Rightarrow \U q\inv=A^*\U.$$
Moreover $\vert q\vert=1$ implies $q\inv=\bbar q${.}
\end{proof}

Let $S=\frac{1}{2}(A+A^*)$ be the Hermitian part of $A$.

\begin{cor}\label{PASO}If $q$ is {a right eigenvalue} of $A$ and $\U$ is a $q$-eigenvector, then $\Re(q)$ is an eigenvalue of $S$, and $\U$ is an $\Re(q)$-eigenvector.
\end{cor}

\begin{cor}The Rayleigh functions of $A$ and $S$ are equal,  $h_A=h_S$.
\end{cor}

\begin{proof}
We have
$$h_S(\V)=\V^*S\V=\frac{1}{2}(\V^*A\V+\V^*A^*\V)= \Re(\V^*A\V)=h_A(\V). \qedhere$$
\end{proof}

\begin{cor}\label{REALEIGEN}
\begin{enumerate}
\item
The $n$ {right} {eigenvalues}  of $S$ are the real parts {$t_j=\Re(q_j)$} of the $n$ similarity classes $[q_1],\dots,[q_n]$ of the {right} eigenvalues of $A$.
\item
The critical values of $h_A$ are $\Re(q_1),\dots,\Re(q_n)$.
\end{enumerate}
\end{cor}

\subsection{Left eigenvalues}

Now, assume that  $\lambda$ is a left eigenvalue of the symplectic matrix $A$. {We know that  it may exist an infinite number of them}.

\begin{prop}  If $\V$ is a {$\lambda$-eigenvector}, that is, $A\V=\lambda\V$, then $h_A(\V)=\Re(\lambda)$.
\end{prop} 

\begin{proof}The proof is identical to that of {Lemma}  \ref{REALPART}. Since $A$ is normal, it is diagonalizable  and we can take an orthonormal basis $\U_1\dots,\U_n$ of eigenvectors \cite[Theorem 5.3.6]{RODMAN}.
By taking  coordinates  $\V=\sum_j \U_jx_j$, we can assume that $\vert \V \vert^2=\sum_j\vv x_j\vv^2=1$, then
$$h_A(\V)=\Re(\V^*\lambda\V)=\Re(\sum_j \bbar x_j \lambda x_j)=\sum_j\Re(\lambda\vv x_j\vv^2)=\Re(\lambda). \qedhere$$

\end{proof}

\begin{nota}Since $\vv \lambda \vv=1$, notice that $A\V=\lambda\V$ implies 
$A^*(\lambda\V)=\bbar \lambda (\lambda\V)$,
hence  $\lambda\V$ is a $\bbar \lambda$-eigenvector of $A^*$, but we cannot conclude nothing about the eigenvalues of
$S$, as Example \ref{FINAL} shows.
\end{nota}

\begin{cor}\label{XX}Let $A$ be a symplectic matrix whose right eigenvalues are organized in $n$ similarity classes  $[q_1],\dots,[q_n]$, ordered in such a way that
$\Re(q_1)\leq \dots\leq \Re(q_n)$. Let  $\lambda$ be a left {eigenvalue} of $A$, such that its eingenspace verifies  $\dim V(\lambda)\geq  k$. {Then}
  $$\Re(q_k)\leq \Re(\lambda) \leq \Re(q_{n-k+1}).$$
\end{cor}

\begin{proof}
Take $\V\in V_A(\lambda)$, so  
$\Re(\lambda)=h_A(\V)=h_S(\V)$. This does not mean that   $\lambda$ is a left eigenvalue of $S$ (see Example \ref{FINAL}). But  $\Re(\lambda)=h_S(\V)$ is the constant value of $h_S$ in the subspace $\E=V_A(\lambda)$,
hence $m_\E=\Re(\lambda)=M_\E$.  If $\dim \E=j\geq  k$, then 
$$\Re(q_k)\leq \Re(q_j)\leq \Re(\lambda)\leq \Re(q_{n-j+1})\leq \Re(q_{n-k-1}),$$  from Theorem \ref{MINMAXTEO}, Corollary \ref{MAXMINCOR} {and Corollary \ref{REALEIGEN}}. \qedhere
\end{proof}

{Remember that $\vert \lambda\vert=1$.}

\begin{ejem}\label{MATRIZJ}
Let  
$
A=\frac{\sqrt{2}}{2}\J \begin{bmatrix}
1&-1\cr
1 &1\cr\end{bmatrix},
$
as in Proposition \ref{SYMPLCASE}, with $r=\J$ and $\theta=\pi/4$. We have $\rho=\J$, and the right eigenvalues are {the similarity classes of}
$$q=\cos\theta\,\J \pm \sin\theta\J^2=\frac{\sqrt{2}}{2}(\pm 1+\J).$$
whose real part is
$\Re(q)=\pm \sqrt{2}/{2}$.

On the other hand, the left eigenvalues are
$$\lambda=\cos\theta\,\J+\sin\theta\, \omega\J=\frac{\sqrt{2}}{2}(1+\omega)\J,$$ 
{where $\Re(\omega)=0$ and $\vv \omega \vv=1$}, {and their real part is
$\Re(\lambda)=\frac{\sqrt{2}}{2}\Re(\omega\J)$.}

Then, since $\vv \omega\J\vv=1$, it is true that $-1\leq \Re(\omega\J)\leq 1$,
hence  $\Re(q_1)\leq \Re(\lambda) \leq \Re(q_2)$.
\end{ejem}

\begin{ejem}\label{FINAL}
For the symplectic matrix $A$ in Example \ref{MATRIZJ}, the Hermitian part is
$S=\frac{\sqrt{2}}{2}\, \J\,\begin{bmatrix}
0&-1\cr
1 & 0\cr
\end{bmatrix}$.
Its right eigenvalues are
$q=\pm \frac{\sqrt{2}}{2}$,  
{which is the real part} of those of $A$, as stated in Corollary \ref{PASO}.

Its left eigenvalues are, by \eqref{TODOS}, 
$$\lambda=\frac{\sqrt{2}}{2}\J\omega,  \quad \Re(\omega)=1, \vv \omega\vv=1,$$
which are different from those of $A$.
\end{ejem}

{Now, we shall prove  a result, analogous to Corollary \ref{SIZE}, showing that the existence of left eigenvalues with a  high-dimensional  space of eigenvectors depends on the multiplicity of the right eigenvalues.}

\begin{cor}Let $A$ be {an} $n\times n$ symplectic matrix. Assume that there is a left eigenvalue $\lambda$  such that $k=\dim V(\lambda) >\ceil{n/2}$. Then,  {the similarity classes of  right eigenvalues $[q_{n-k+1}]=\dots=[q_k]$   have multiplicity $l_k\geq  2k-n$, and they
are equal to the similarity class $[\lambda]$ of $\lambda$.}\end{cor}

\begin{proof}
Remember that the eigenvalue classes $[q_i]$ are ordered according to their real parts, as in Corollary \ref{XX}. From {that Corollary, we know that 
$$\Re(q_k)\leq \Re(\lambda) \leq \Re(q_{n-k+1}).$$
But $k>\ceil{n/2}$ implies $n-k+1\leq k$, hence $\Re(q_{n-k+1})\leq \Re(q_k)$. Then} 
$\Re(q_{n-k+1})=\dots=\Re(q_k)=\Re(\lambda)$. But since $\vv q_j\vv=1$, for all $j$, and $\vv \lambda \vv=1$,  it follows that the similarity classes are equal,  so we have
$[q_{n-k+1}]=\dots=[q_k]=[\lambda]$.
This proves the result.
\end{proof}


\section*{Funding}
The first two authors are  partially supported by the MINECO and 
FEDER research project MTM2016-78647-P.  The first author was partially supported by Xunta de Galicia ED431C 2019/10 with FEDER funds. 

\bibliographystyle{plain}
\bibliography{left_v2}

\bigskip

\end{document}